\documentclass[twoside,leqno,A4,10pt]{amsart}
\usepackage{amsfonts}
\usepackage{amsmath}
\usepackage{amscd}
\usepackage{amssymb}
\usepackage{amsthm}
\usepackage{amsrefs}
\usepackage{latexsym}
\usepackage{bbm}
\usepackage{manfnt}
\numberwithin{equation}{section}

\begin{document}
\newtheorem*{theorem}{Theorem}
\newtheorem{lemma}{Lemma}
\newtheorem*{corollary}{Corollary}
\numberwithin{equation}{section}
\newcommand{\dif}{\mathrm{d}}
\newcommand{\intz}{\mathbb{Z}}
\newcommand{\ratq}{\mathbb{Q}}
\newcommand{\natn}{\mathbb{N}}
\newcommand{\comc}{\mathbb{C}}
\newcommand{\rear}{\mathbb{R}}
\newcommand{\prip}{\mathbb{P}}
\newcommand{\uph}{\mathbb{H}}
\newcommand{\fief}{\mathbb{F}}
\newcommand{\majorarc}{\mathfrak{M}}
\newcommand{\minorarc}{\mathfrak{m}}
\newcommand{\sings}{\mathfrak{S}}

\title{On Primes Represented by Cubic Polynomials}
\author{Timothy Foo \and Liangyi Zhao}
\date{\today}
\maketitle

\begin{abstract}
In this paper, we prove a theorem on the distribution of primes in cubic progressions on average.
\end{abstract}




\section{Introduction}

It was proved by Dirichlet that any linear polynomial in one variable represents infinitely many primes provided that the coeffiicents are co-prime.  Though it has long been conjectured, analogous statements are not known for any polynomial of degree two or higher.  Many partial results are known in this direction.  The readers are referred to the survey article \cite{BZ2} for more information. \newline

In \cites{SBLZ, BZ}, S. Baier and L. Zhao established certain theorems for the Bateman-Horn conjecture for quadratic polynomials on average.  The latest result in those works states the following.  For all $z >3$ and $B>0$, we have, for $z^{1/2 + \varepsilon}\leq K \leq z/2$
\[ \sum_{1 \leq k \leq K}\left|\sum_{z < n^2+k \leq 2z}\Lambda(n^2+k) -  \prod_{p>2} \left( 1 - \frac{\left(\frac{-k}{p}\right)}{p-1} \right) \sum_{z < n^2+k \leq 2z}1\right|^2 \ll \frac{Kz}{(\log z)^B}, \]
where $\Lambda$ henceforth denotes the von Mangoldt function and $\left( \frac{-k}{p} \right)$ is the Legendre symbol. \newline

The Bateman-Horn conjecture \cite{PTBRAH} says that if $f$ is an irreducible polynomial in $\intz[x]$ satisfying
\[ \gcd\{f(n): n \in \intz\}=1, \]
then
\begin{equation} \label{bhconj}
 \sum_{n \leq x} \Lambda(f(n)) \sim \prod_p \left( 1- \frac{n_p-1}{p-1} \right) x,
 \end{equation}
where $n_p$ is the number of solutions to the equation $f(n) \equiv 0 \pmod{p}$ in $\intz/p \intz$. A little more can be said about the quantity $n_p$ in the specific case where $f(n)=n^d+k$. Specifically, in that case, we have that
\[ n_p = \sum_{i=0}^{d^{\prime}-1}\chi^{i}(-k) \]
where $\chi$ is a Dirichlet character of order $d^{\prime}=\gcd (p-1,d)$. \newline

In this paper, we shall study the asymptotic formula in \eqref{bhconj} on average over a family of cubic polynomials.  Our main theorem is the following.
\begin{theorem}
Given $A,B > 0$, we have, for $x^3 (\log x)^{-A}< y < x^3$,
\[ \sum_{\substack{1 \leq k \leq y\\ \mu^2(k) = 1}}\left| \sum_{n \leq x}\Lambda(n^3 + k) - \sings (k)x\right|^2 = O\left(\frac{yx^2}{(\log x)^B}\right), \]
where $\mu$ henceforth denotes the M\"{o}bius function and
\[ \sings(k) = \prod_{p}\left(1-\frac{n_p-1}{p-1}\right), \]
with $n_p$ being the number of solutions to the equation $n^3+k \equiv 0 \bmod p$ in $\intz/p\intz$.
\end{theorem}

From the theorem, we immediately deduce the following.

\begin{corollary}
Given $A$, $B$, $C>0$ and $\sings (k)$ as defined in the theorem, we have, for $x^3 (\log x)^{-A} \leq y \leq x^3$, that
\[ \sum_{n \leq x} \Lambda (n^3+k) = \sings(k) x + O \left( \frac{x}{(\log x)^B} \right) \]
holds for all square-free $k$ not exceeding $y$ with at most $O( y (\log x)^{-C})$ exceptions.
\end{corollary}

We shall use the circle method to study this problem.  The readers are referred to \cite{RCV} or Chapter 20 of \cite{HIEK} for an introduction to the circle method.  The starting point for us is the identity
\begin{equation} \label{startpt}
 \sum_{n \leq x}\Lambda(n^3+k) = \int\limits_0^1 \sum_{m \leq z}\Lambda(m)e(\alpha m)\sum_{n\leq x}e(-\alpha (n^3+k)) \dif \alpha,
 \end{equation}
where henceforth $z=x^3+y$ and $e(w) = \exp(2 \pi i w)$. \newline

\section{Preliminary Lemmas}

In this section, we collect some preliminary lemmas which we shall need in the sequel.

\begin{lemma} [Gallagher] \label{gallalem}
Let $2 < \Delta < N/2$ and $N < N^{\prime} < 2N $. For arbitrary $a_n \in \comc$, we have
$$
\int\limits_{|\beta| < \Delta^{-1}} \left| \sum_{N<n<N^{\prime}}a_n e(\beta n)\right|^2 \dif \beta \ll \Delta^{-2}\int\limits_{N-\Delta/2}^{N^{\prime}}\left|\sum_{\max(t,N)<n<\min(t+\Delta/2,N^{\prime})}a_n\right|^2 \dif t
$$
where the implied constant is absolute.
\end{lemma}

\begin{proof}This is Lemma 1 in \cite{Galla} in a slightly modified form. \end{proof}

We shall also need the following lemma.

\begin{lemma} [Wolke, Mikawa] \label{mikawalem}
Let
$$ \mathfrak{J}(q,\Delta) = \sum_{\chi \bmod q}\int\limits_N^{2N} \left|\sum_{t < n <t + q\Delta}^{\sharp}\chi(n)\Lambda(n) \right|^2 \dif t
$$
\end{lemma}
\noindent
where the $\sharp$ over the summation symbol means that if $\chi=\chi_0$, then $\chi(n)\Lambda(n)$ is replaced by $\Lambda(n)-1$. Let $\varepsilon$, $A$ and $B>0$ be given. If $q \leq (\log N)^B$ and $ N^{1/5 + \varepsilon}<\Delta <N^{1-\varepsilon}$, then we have

\begin{equation}\label{3} \mathfrak{J}(q,\Delta) \ll (q\Delta)^2 N(\log N)^{-A}
\end{equation}
where the implied constant depends only on $\varepsilon$, $A$ and $B$.

\begin{proof}This is Lemma 2 in \cite{Mika}.\end{proof}

The following Bessel's lemma plays an important role in many places of the proof.

\begin{lemma} [Bessel] \label{bessineq}
Let $\phi_1,\phi_2,\dots,\phi_R$ be orthonormal members of an inner product space $V$ over $\comc$ and let $\xi \in V$. Then
\[ \sum_{r=1}^R \left| (\xi,\phi_r) \right|^2 \leq (\xi , \xi). \]
\end{lemma}

\begin{proof} This is a standard result. See \cite{PRH} for a proof. \end{proof}

The following lemma is needed in the estimates involving the minor arcs.

\begin{lemma} [Weyl Shift] \label{weylshift}
If $f(x) = \alpha x^k + \dots$ is a polynomial with real coefficents and $k \geq 1$, then
\[ \left| \sum_{1 \leq n \leq N}e(f(n)) \right| \leq 2N\left\{N^{-k}\sum_{-N<l_1,l_2,\dots,l_{k-1}<N}\min \left( N,\frac{1}{\left\| \alpha k!\prod_{i=1}^{k-1}l_i \right\|} \right)\right\}^{2^{1-k}}, \]
where $\| x \|$ denotes the distance of $x \in \rear$ to the closest integer.
\end{lemma}

\begin{proof} This is Proposition 8.2 of \cite{HIEK}. \end{proof}

To dispose of the contributions from certain parts of the singular series, $\sings (k)$, we shall need the following lemmas.

\begin{lemma} [P\'olya-Vinogradov] \label{polyvino}
For any non-principle character $\chi$ modulo $q$, we have
\[ \left|\sum_{M<n\leq M+N}\chi(n)\right| \leq 6\sqrt{q}\log q. \]
\end{lemma}
\begin{proof} This is Theorem 12.5 in \cite{HIEK}. \end{proof}

\begin{lemma} [Baier-Young] \label{B&Y}
For $(a_m)_{m \in \intz}$ with $a_m \in \comc$ arbitrary, we have
$$ \sum_{Q<q<2Q} \ \sideset{}{^{\star}}\sum_{\substack{\chi \bmod q\\ \chi^3 = \chi_0}}\left| \ \sideset{}{^{*}}\sum_{M<m<2M} a_m\chi(m)\right|^2 \ll (QM)^{\varepsilon} \left( Q^{11/9} + Q^{2/3}M \right) \sideset{}{^{*}} \sum_{\substack{M<m<2M\\(m,3)=1}} |a_m|^2,
$$
where the $*$ over the sum over $m$ indicates that the sum is over squarefree integers, the $\star$ over the sum over $\chi$ indicates that $\chi$ is primitive.
\end{lemma}

\begin{proof} This is one of the bounds for the mean-values of cubic character sums in Theorem 1.4 in \cite{BaYo}. \end{proof}

\begin{lemma} [Huxley] \label{huxleylemma}
Let $K$ be a number field and $\mathfrak{r}$ denote an ideal in $K$.  Suppose $u(\mathfrak{r})$ is a complex-valued function defined on the set of ideals in $K$.  We have
\[ \sum_{\mathcal{N}(\mathfrak{f})\leq Q}\frac{\mathcal{N}(\mathfrak{f})}{\Phi(\mathfrak{f})} \sideset{}{^{\star}}\sum_{\chi \bmod{\mathfrak{f}}} \left|  \sum_{\mathcal{N}(\mathfrak{r})\leq z}u(\mathfrak{r})\chi(\mathfrak{r}) \right|^2 \ll (z+Q^2)\sum_{\mathcal{N}(\mathfrak{r}) \leq z}|u(\mathfrak{r})|^2, \]
where $\mathcal{N}(\mathfrak{f})$ denotes the norm of the ideal $\mathfrak{f}$, $\Phi(\mathfrak{f})$ is Euler's totient function generalized to the setting of number fields, the $\star$ over the sum over $\chi$ indicates that $\chi$ is a primitive character of the narrow ideal class group modulo $\mathfrak{f}$ and the implicit constant depends on $K$.
\end{lemma}

\begin{proof} This is Theorem 1 of \cite{Hux}. \end{proof}

\begin{lemma}[Duality Principle]\label{dual}
Let $T=[t_{mn}]$ be a finite square matrix with entries from the complex numbers.  The following two statements are equivalent:
\begin{enumerate}
\item For any complex numbers $\{ a_n \}$, we have
\begin{equation*}
\sum_m \left| \sum_n a_n t_{mn} \right|^2 \leq D \sum_n |a_n|^2.
\end{equation*}
\item For any complex numbers $\{ b_n \}$, we have
\begin{equation*}
\sum_n \left| \sum_m b_m t_{mn} \right|^2 \leq D \sum_m |b_m|^2.
\end{equation*}
\end{enumerate}
\end{lemma}
\begin{proof} This is a standard result.  See Theorem 228 in \cite{GHHJELGP}. \end{proof}

\begin{lemma} \label{zerofree}
Let $K/\ratq$ be a number field, $\xi$ a Hecke Grossencharakter modulo $(\mathfrak{m}, \Omega)$ where $\mathfrak{m}$ is a non-zero integral ideal in $K$ and $\Omega$ is a set of real infinite places where $\xi$ is ramified. Let the conductor $\Delta = |d_{K}|N_{K/\ratq}\mathfrak{m}$. There exists an absolute effective constant $c'>0$ such that the L-function $L(\xi,s)$ of degree $d = [K:\ratq]$ has at most a simple real zero in the region
\[ \sigma > 1 - \frac{c'}{d\log \Delta(|t|+3)}. \]
The exceptional zero can occur only for a real character and it is $< 1$.
\end{lemma}
\begin{proof} This is Theorem 5.35 of \cite{HIEK}. \end{proof}

\begin{lemma} [Perron]\label{perron}
Suppose that $y \not= 1$ is a positive real number and let
$$\delta(y) = \begin{cases}
  1,  & y > 1, \\
  & \\
  0, & \mbox{otherwise.}
  \end{cases}
$$
Furthermore, let $c>0$ and $T>0$. Then
$$ \frac{1}{2 \pi i}\int_{c-iT}^{c+iT}\frac{y^s}{s}ds = \delta(y) + O\left(y^c \min\left(1, T^{-1}\left|\log y\right|^{-1}\right)\right).
$$
\end{lemma}
\begin{proof} This is a standard result. See Chapter 17 in \cite{HD}. \end{proof}

\section{The Major Arcs}

The major arcs are defined by
\begin{equation} \label{majorarcdef}
 \majorarc = \bigcup_{q \leq Q_1}\bigcup_{\substack{a=1 \\ \gcd (a,q) = 1}}^{q} J_{q,a}
\end{equation}
where
\[ J_{q,a} = \left[\frac{a}{q}-\frac{1}{qQ},\frac{a}{q}+\frac{1}{qQ} \right] , \; Q_1 = (\log x)^c \; \; \mbox{and} \; \; Q = x^{1 - \varepsilon}, \]
for some $c> 0$ fixed and suitable.  Note that if $x$ is sufficiently large, $Q> Q_1$ and the intervals $J_{q,a}$ are disjoint.  We shall henceforth assume that this is the case. \newline

For $\alpha \in \majorarc$, we write
\[ \alpha = \frac{a}{q} + \beta, \; \mbox{with} \; |\beta| \leq \frac{1}{qQ}. \]

Now set
\[ S_1(\alpha) = \sum_{m \leq z}\Lambda(m)e(\alpha m) \; \; \; \mbox{and} \; \; \; S_2(\alpha) = \sum_{n \leq x}e(-\alpha n^3). \]

$S_1(\alpha)$ has been computed in Section 3 of \cite{SBLZ} and we have
\[ S_1( \alpha ) = T_1(\alpha) + E_1(\alpha) + O \left( \log^2 z \right) \]
where
\begin{equation} \label{T1E1def}
 T_1(\alpha) = \frac{\mu(q)}{\varphi(q)}\sum_{m \leq z}e(\beta m) \;\;\; \mbox{and} \;\;\; E_1(\alpha) = \frac{1}{\varphi(q)} \sum_{\chi \bmod{q}} \tau(\overline{\chi}) \chi(a) \sum_{m\leq z}^{\sharp} \chi(m) \Lambda(m) e (\beta m)
 \end{equation}
with $\varphi$ being the Euler totient function,
\begin{equation} \label{gausssumdef}
 \tau(\chi) = \sum_{r \bmod q}\chi(r)e\left(\frac{r}{q}\right)
\end{equation}
the Gauss sum and the $\sharp$ of the summation symbol in $E_1(\alpha)$ having the same meaning as that in Lemma~\ref{mikawalem}. \newline

Now to compute $S_2(\alpha)$, we first note the following well-known relations between additive and multiplicative characters.
\begin{equation} \label{addmultcharrel}
e \left( \frac{a}{q} m \right) = \frac{1}{\varphi(q)} \sum_{\chi \bmod{q}} \chi(am) \tau(\overline{\chi})
\end{equation}
provided that $\gcd (am, q) =1$.  Therefore, using \eqref{addmultcharrel}, we have
\[ S_2(\alpha) = \sum_{n \leq x}e\left(-\left(\frac{a}{q} + \beta\right)n^3\right) = \sum_{d|q}\frac{1}{\varphi(q_1^*)}\sum_{\chi \bmod q_1^*}\chi(-ad^*)\tau(\bar{\chi})\sum_{\substack{n\leq x\\ \gcd(n,q)=d}}\chi^3(n^{*})e(-\beta n^3), \]
where
\[ q^* = \frac{q}{d}, \; n^* = \frac{n}{d}, \; d^* = \frac{d^2}{\gcd(d^2, q^*)} \; \;
\mbox{and} \; \; q_1^* = \frac{q^*}{\gcd(d^2, q^*)}. \]
So
\begin{eqnarray*}
 S_2(\alpha) &=& \sum_{d|q}\frac{1}{\varphi(q_1^*)}\sum_{\substack{\chi \bmod q_1^*\\ \chi^3  = \chi_0}}\chi(-ad^*)\tau(\bar{\chi})\sum_{\substack{n\leq x\\ \gcd(n,q)=d}}e(-\beta n^3)\\
 && \hspace*{1cm}  +\sum_{d|q}\frac{1}{\varphi(q_1^*)}\sum_{\substack{\chi \bmod q_1^*\\ \chi^3 \not= \chi_0}}\chi(-ad^*)\tau(\bar{\chi})\sum_{\substack{n\leq x\\ \gcd(n,q)=d}}\chi^3(n^{*})e(-\beta n^3)\\
&=& T_2(\alpha) + E_2(\alpha),
\end{eqnarray*}
say.  Now performing some computations similar to those in Section 3 of \cite{SBLZ}, we have
\[ T_2(\alpha) = \sum_{d|q}\frac{1}{\varphi(q_1^*)} \sum_{\substack{l \bmod q_1^* \\ \gcd(l,q_1^*)=1}}e\left(-\frac{ad^*l^3}{q_1^*}\right)\sum_{\substack{n\leq x\\ \gcd(n,q)=d}}e(-\beta n^3). \]
Therefore,
\begin{equation} \label{majorarc1}
\begin{split}
\int\limits_{\majorarc} \sum_{m \leq z} & \Lambda(m)e(\alpha m)\sum_{n\leq x}e(-\alpha (n^3+k)) \dif \alpha \\
& = \int\limits_{\majorarc} \left( T_1(\alpha) + E_1( \alpha) + O \left( \log^2 x \right) \right) \left( T_2(\alpha) + E_2(\alpha) \right) e(-\alpha k) \dif \alpha .
\end{split}
\end{equation}

\section{The Singular Series}

The main term will be given by the following
\begin{equation*}
\begin{split}
 \int\limits_{\majorarc}T_1(\alpha)& T_2(\alpha)e(-k\alpha) \dif \alpha \\
& = \sum_{q \leq Q_1}\frac{\mu(q)}{\varphi(q)}\sum_{\substack{a \bmod q \\ \gcd(a,q)=1}}e\left(-\frac{ak}{q}\right)\sum_{d|q}\frac{1}{\varphi(q_1^*)}\sum_{\substack{l \bmod q_1^*\\ \gcd(l,q_1^*)=1}}e\left(-\frac{ad^*l^3}{q_1^*}\right)\int\limits_{|\beta| < \frac{1}{qQ}}\Pi_{q,d}(\beta) \dif \beta
\end{split}
\end{equation*}
where
\[ \Pi_{q,d}(\beta) = \sum_{m \leq z}e(\beta m)\sum_{\substack{n \leq x\\ \gcd(n,q)=d}}e(-\beta n^3)e(-\beta k). \]
As in \cite{SBLZ}, we have
\[  \int\limits_{|\beta| < \frac{1}{qQ}}\Pi_{q,d}(\beta) \dif \beta = \frac{\varphi(q/d)}{q} x + O\left(\left(\frac{qQx}{d}\right)^{1/2} \right). \]
Hence
\begin{equation} \label{T1T2-2}
\begin{split}
\int\limits_{\majorarc} & T_1(\alpha) T_2(\alpha)e(-k\alpha) \dif \alpha \\
&= \sum_{q \leq Q_1}\frac{\mu(q)}{\varphi(q)}\sum_{\substack{a \bmod q\\ \gcd(a,q)=1}}e\left(-\frac{ak}{q}\right)\sum_{d|q}\frac{1}{\varphi(q_1^*)} \sum_{\substack{l \bmod q_1^*\\ \gcd(l,q_1^*)=1}}e\left(-\frac{ad^*l^3}{q_1^*}\right)\left(\frac{\varphi(q/d)}{q}x+O\left(\left(\frac{qQx}{d}\right)^{1/2}\right)\right)
\end{split}
\end{equation}
Due to the presence of $\mu(q)$, it suffices to consider only the $q$'s that are square-free.  Therefore, we immediately get $d^* = d^2$ and $q_1^* = q^* = q/d$.  Thus the right-hand side of \eqref{T1T2-2} becomes
\begin{equation}\label{2}
 x\sum_{q \leq Q_1}\frac{\mu(q)}{\varphi(q)q}\sum_{\substack{a \bmod q\\ \gcd(a,q)=1}}e\left(-\frac{ak}{q}\right)\sum_{d|q}\sum_{\substack{l \bmod q/d\\ \gcd(l,q/d)=1}}e\left(-\frac{a (dl)^3}{q}\right)+ O\left( \sqrt{xQ}(\log x)^{c_1} \right).
\end{equation}
We observe that
\begin{equation} \label{Sigmadef}
 \Sigma(q) := \sum_{\substack{a \bmod q\\ \gcd(a,q)=1}}e\left(-\frac{ak}{q}\right)\sum_{d|q}\sum_{\substack{l \bmod q/d\\ \gcd(l,q/d)=1}}e\left(-\frac{a(dl)^3}{q}\right) = \sum_{\substack{a \bmod q\\ \gcd(a,q)=1}}e\left(-\frac{ak}{q}\right) \sum_{r \bmod q}e\left(-\frac{ar^3}{q}\right).
 \end{equation}
Suppose that $q = p$, a prime number.  Then
\[ \Sigma(p) = \sum_{\substack{a \bmod p\\ \gcd(a,p)=1}}\sum_{r \bmod p}e\left(\frac{-a(k+r^3)}{p}\right) = (p-1)n_{k,p} + (-1)(p - n_{k,p}) = p(n_{k,p} - 1) \]
where $n_{k,p}$ is the number of solutions to $k+r^3 \equiv 0 \pmod{p}$ in $\intz/p\intz$.
Therefore, letting $\chi_{1,p}$ and $\chi_{2,p}$ be the two cubic characters not equal to $\chi_0$ for a prime $p \equiv 1$ mod $3$,
\[ \Sigma(p) = \begin{cases}
  p \left( \chi_{1,p}(-k) + \chi_{2,p}(-k) \right),  & \mbox{if }p \equiv 1 \pmod{3} \\
  & \\
  0, & \mbox{if }p = 3 \mbox{ or } p \equiv 2 \pmod{3}
  \end{cases} \]
since the map $r \longrightarrow r^3$ is a bijection on $\intz/p\intz$ when $p = 3$ or $p \equiv 2 \pmod{3}$. Furthermore, $\Sigma(q)$ is multiplicative in $q$.  To see this, we have if $\gcd(q_1,q_2)=1$, then
 \[ \Sigma(q_1)\Sigma(q_2) = \sum_{r_1\bmod q_1}\sum_{\substack{a_1 \bmod q_1\\(a_1,q_1)=1}}\sum_{r_2\bmod q_2}\sum_{\substack{a_2 \bmod q_2\\(a_2,q_2)=1}}e\left( f(k,a_1,a_2,q_1,q_2,r_1,r_2\right) \]
where
\begin{equation*}
\begin{split}
 f(k,a_1,a_2,q_1,q_2,r_1,r_2) & = -k\frac{a_1q_2+a_2q_1}{q_1q_2}-\frac{a_1q_2(q_2r_1)^3+a_2q_1(q_1r_2)^3}{q_1q_2} \\
 & \equiv -k \frac{a_1q_2+a_2q_1}{q_1q_2} - \frac{(a_1q_2+a_2q_1)(q_1r_2+q_2r_1)^3}{q_1q_2} \pmod{1}.
 \end{split}
 \end{equation*}
 It is easy to observe that $a_1q_2+a_2q_1$ runs over the primitive residue classes modulo $q_1q_2$ as $a_1$ and $a_2$ run over the primitive residue classes modulo $q_1$ and $q_2$, respectively.  The same can be said for $q_2r_1+q_1r_2$.  Hence
 \[ \Sigma(q_1) \Sigma(q_2) = \Sigma(q_1q_2). \]
 Therefore, we have, for $q$ squarefree,
\begin{equation} \label{Sigmaeval}
 \Sigma(q) = \prod_{p|q}p(n_{k,p} - 1) = \begin{cases}
  \displaystyle{q\prod_p(\chi_{1,p}(-k) + \chi_{2,p}(-k))},  & \mbox{if }p \equiv 1 \pmod{3} \mbox{ for all } p|q , \\
  & \\
  0, & \mbox{otherwise.}
  \end{cases}
  \end{equation}

Hence, combining \eqref{T1T2-2}, \eqref{2}, \eqref{Sigmadef} and \eqref{Sigmaeval} yields
\begin{equation} \label{mainterm}
\begin{split}
\int\limits_{\majorarc} T_1(\alpha) T_2(\alpha)e(-k\alpha) \dif \alpha & =
x\sum_{q \leq Q_1}\frac{\mu(q)}{\varphi(q)q} \Sigma(q) + O\left( \sqrt{Qx}(\log x)^{c_1} \right) \\
&= x\sum_{q =1 }^{\infty}\frac{\mu(q)}{\varphi(q)} \prod_{p|q}(n_{k,p}-1) - x\Psi(k) + O\left( \sqrt{Qx}(\log x)^{c_1} \right) \\
&= \mathfrak{S}(k)x + O\left( |\Psi(k)|x + \sqrt{Qx}(\log x)^{c_1} \right)
\end{split}
\end{equation}
where
\begin{equation} \label{Psikdef}
 \Psi(k) = \sum_{q > Q_1}\frac{\mu(q)}{\varphi(q)} \prod_{p|q} \left( n_{k,p}-1 \right).
 \end{equation}

We note that the infinite sums appearing above are conditionally convergent, and this is shown later section 5.

\section{The second moment of $\Psi(k)$}

We aim to estimate the second moment of $\Psi(k)$, defined in \eqref{Psikdef}, in this section.  We shall assume that all $q$'s appearing in the subsequent computations are divisible only by primes $p \equiv 1$ mod $3$ and are squarefree.  Dividing the $q$-range into three pieces, we have
\begin{equation} \label{splitPhi}
\sum_{\substack{k \leq y \\ \mu(k)^2 = 1}} |\Psi(k)|^2 \ll  \Psi_1 + \Psi_2 + \Psi_3,
\end{equation}
where
\[ \Psi_1 = \sum_{k \leq y} \left|\sum_{Q_1 < q \leq U}\frac{\mu(q)}{\varphi(q)}\prod_{p|q} \left(\chi_{1,p}(-k)+\chi_{2,p}(-k) \right)\right|^2, \]
\[ \Psi_2 = \sum_{\substack{k \leq y \\ \mu(k)^2 = 1}}\left|\sum_{U < q \leq 2^vU}\frac{\mu(q)}{\varphi(q)}\prod_{p|q}(\chi_{1,p}(-k)+\chi_{2,p}(-k))\right|^2, \]
and
\[ \Psi_3 = \sum_{k \leq y}\left|\sum_{q > 2^vU}\frac{\mu(q)}{\varphi(q)}\prod_{p|q}(\chi_{1,p}(-k)+\chi_{2,p}(-k))\right|^2 \]
with $U$ and $v$ being parameters to be chosen later. \newline

Expanding the modulus square, we have
\begin{equation} \label{Psi1expand}
\begin{split}
\Psi_1 \leq y& \sum_{Q_1 < q \leq U} \frac{\mu(q)^2}{\varphi(q)^2}2^{2\omega(q)}  \\
& + \sum_{\substack{Q_1 < q_1,q_2 \leq U\\q_1 \not= q_2}}\frac{\mu(q_1)\mu(q_2)}{\varphi(q_1)\varphi(q_2)}\sum_{k \leq y}\prod_{p|q_1}(\chi_{1,p}(-k)+\chi_{2,p}(-k))\prod_{p|q_2}(\chi_{1,p}(-k)+\chi_{2,p}(-k))
\end{split}
\end{equation}
where $\omega(q)$ is the number of distinct primes dividing $q$. \newline

Mindful of the well-known estimate
\begin{equation} \label{phiest}
 \varphi(q) \gg \frac{q}{\log \log (10 q)},
 \end{equation}
the first term on the right-hand side of \eqref{Psi1expand} can be easily disposed.  We have
\begin{equation} \label{Psi1diag}
y\sum_{Q_1 < q \leq U}\frac{\mu(q)^2}{\varphi(q)^2}2^{2\omega(q)} \ll \frac{y}{Q_1^{1-\varepsilon}}.
\end{equation}
The second term on the right-hand side of \eqref{Psi1expand} is
\begin{equation} \label{Psi1offdiag}
\begin{split}
\sum_{\substack{Q_1 < q_1,q_2 \leq U\\q_1 \not= q_2}} & \frac{\mu(q_1)\mu(q_2)}{\varphi(q_1)\varphi(q_2)} \sum_{k \leq y}\prod_{p|q_1}(\chi_{1,p}(-k)+\chi_{2,p}(-k))\prod_{p|q_2}(\chi_{1,p}(-k)+\chi_{2,p}(-k)) \\
&\ll \sum_{\substack{Q_1 < q_1,q_2 \leq U\\q_1 \not= q_2}}\frac{\sqrt{q_1q_2}\log(q_1q_2)2^{\omega(q_1)+\omega(q_2)}}{\varphi(q_1)\varphi(q_2)} \ll \log U \left(\sum_{Q_1 < q \leq U}\frac{\sqrt{q}2^{\omega(q)}}{\varphi(q)}\right)^2 \ll U^{1 + \varepsilon}
\end{split}
\end{equation}
using the Polya-Vinogradov inequality, Lemma~\ref{polyvino}, and \eqref{phiest}.  Note the resulting characters after expanding the products on the left-hand side of \eqref{Psi1offdiag} for are of modulus $q_1q_2$ and non-principle since $q_1 \neq q_2$. \newline

Combining \eqref{Psi1expand}, \eqref{Psi1diag} and \eqref{Psi1offdiag}, we arrive at
\begin{equation} \label{Psi1est}
\Psi_1 \ll \frac{y}{Q_1^{1-\varepsilon}} + U^{1+\varepsilon}.
\end{equation}

To estimate $\Psi_2$, we first use Cauchy's inequality and get
\begin{equation} \label{Psi2aftercauchy}
 \Psi_2 \ll v\sum_{r=1}^{[v+1]}\sum_{\substack{k \leq y \\ \mu(k)^2 = 1}}\left|\sum_{2^{r-1}U < q \leq 2^rU}\frac{\mu(q)}{\varphi(q)}\prod_{p|q}(\chi_{1,p}(-k)+\chi_{2,p}(-k))\right|^2.
 \end{equation}
Now expanding the product over $p$ inside the modulus signs in \eqref{Psi2aftercauchy}, we have
\begin{equation} \label{chiexpand}
\sum_{\substack{k \leq y \\ \mu(k)^2 = 1}}\left|\sum_{2^{r-1}U < q \leq 2^rU}\frac{\mu(q)}{\varphi(q)}\prod_{p|q}(\chi_{1,p}(-k)+\chi_{2,p}(-k))\right|^2
 = \sum_{\substack{k \leq y \\ \mu(k)^2 = 1}}\left| \sum_{2^{r-1}U < q \leq 2^rU} \ \sideset{}{^{\star}}\sum_{\substack{\chi \bmod q\\\chi^3 = \chi_0\\\chi \not= \chi_0}}\frac{\mu(q)}{\varphi(q)}\chi(k)\right|^2
\end{equation}
Note that all the resulting $\chi$'s appearing on the right-hand side of \eqref{chiexpand} are primitive modulo $q$.  Using Lemma~\ref{B&Y} together with the duality principle, Lemma~\ref{dual}, the right-hand side of \eqref{chiexpand} is
\[ \ll y^{\varepsilon} \left( 2^{r-1}U \right)^{2/9 + \varepsilon} + y^{1 + \varepsilon}\left(\frac{1}{2^{r-1}U}\right)^{1/3 - \varepsilon}. \]
Summing the above over $r$ from $1$ to $R =[\log_2 \left( (y^{3 + 3\varepsilon/2})/U \right) ]$, we get
\begin{equation} \label{Psi2est1}
 \sum_{r=1}^{R}\sum_{\substack{k \leq y \\ \mu(k)^2 = 1}}\left|\sum_{2^{r-1}U < q \leq 2^rU}\frac{\mu(q)}{\varphi(q)}\prod_{p|q}\left( \chi_{1,p}(-k)+\chi_{2,p}(-k) \right)\right|^2 \ll y^{2/3 + \varepsilon} + \frac{y^{1 + \varepsilon}}{U^{1/3-\varepsilon}}.
 \end{equation}
For the $r$'s with $R < r \leq [v+1]$, we shall use the large sieve for number fields, Lemma~\ref{huxleylemma}.  It is easy to reduce the expression in question to a sum of similar shape with the additional summation condition $\gcd (k, 3)=1$ included.  Also recall that $q$ is assumed to be square-free and have only prime factors that are congruent to 1 modulo 3.  Hence it suffices to estimate
\begin{equation}\label{1}
 \sum_{\substack{n \in \intz[\omega] \\ \mathcal{N}(n) \leq y^2 \\ n \equiv 1 \bmod{3} \\ \mu(n)^2=1}} \left| \sum_{\substack{\pi \in \mathbb{Q}[\omega]\\2^{r-1}U < N(\pi) = q \leq 2^rU}} a_{\pi} \chi_{\pi}(n)\right|^2,
\end{equation}
where $\omega = e(1/3)$ here and after, $a_{\pi} = \mu(q)/\varphi(q)$, $\chi_{\pi}(n) = \left( \frac{n}{\pi} \right)_3$ where $( \frac{\cdot}{\cdot} )_3$ is the cubic residue symbol.  Recall that by cubic reciprocity $\chi_{\pi}(n) = \chi_n(\pi)$ for all the $n$ and $\pi$ appearing above.  Using this, Lemma~\ref{huxleylemma} and \eqref{phiest}, we have that \eqref{1} is majorized by
\[ \ll (y^4 + 2^rU)\sum_{\substack{\mathcal{N}(\pi) = q\\2^{r-1}U < q \leq 2^rU}}\frac{1}{\varphi(q)^2} \ll (y^4 + 2^rU) \frac{\log\log(2^rU)\log(2^rU)}{2^rU}. \]

Therefore,
\[ \sum_{\substack{k \leq y \\ \mu(k)^2 = 1}}\left|\sum_{2^{r-1}U < q \leq 2^rU}\frac{\mu(q)}{\varphi(q)}\prod_{p|q}(\chi_{1,p}(-k)+\chi_{2,p}(-k))\right|^2 \ll \left( 2^rU + y^4 \right)\frac{\log\log(2^rU)\log(2^rU)}{2^rU}. \]
Summing the above over $r$ from $R$ to $[v+1]$, we have
\begin{equation} \label{Psi2est2}
\begin{split}
\sum_{r=R}^{[v+1]} \sum_{\substack{k \leq y \\ \mu(k)^2 = 1}} & \left|\sum_{2^{r-1}U < q \leq 2^rU}\frac{\mu(q)}{\varphi(q)}\prod_{p|q}(\chi_{1,p}(-k)+\chi_{2,p}(-k))\right|^2 \\
&\ll v^3 + v^2\log U + v(\log U)^2 + y^{1-3\varepsilon/2}\left( v+ \log U \right) \log(v+ \log U).
\end{split}
\end{equation}

Therefore, \eqref{Psi2est1} and \eqref{Psi2est2} give that
\begin{equation} \label{Psi2est}
\Psi_2 \ll v\left(y^{2/3 + \varepsilon} + \frac{y^{1 + \varepsilon}}{U^{1/3-\varepsilon}} + v^3 + v^2\log U + v(\log U)^2 + y^{1-3\varepsilon/2}(v+ \log U)\log(v+ \log U)\right).
\end{equation}

It still remains to consider $\Psi_3$.  First, note that for primes $p \equiv 1 \bmod 3$, these primes split into prime ideals in $\mathbb{Q}[ \omega ]$ as $p = \pi_{p,1}\pi_{p,2}$. Let us consider
\begin{equation*}
\begin{split}
 f(s,k) & = \prod_{p \equiv 1 \bmod 3}\left(1-\frac{\chi_{1,p}(-k) + \chi_{2,p}(-k)}{(p-1)p^s}\right) = \prod_{p \equiv 1 \bmod 3}\left(1-\frac{n_{k,p} - 1}{(p-1)p^s}\right) \\
& = \sum_{q}\frac{\mu(q)}{\varphi(q)q^s} \prod_{p|q}(n_{k,p}-1) = \prod_{\substack{p \equiv 1 \bmod 3\\ p = \pi_{p,1}\pi_{p,2}}}\left(1-\frac{\left(\frac{k}{\pi_{p,1}}\right)_3+\left(\frac{k}{\pi_{p,2}}\right)_3}{(p-1)p^s}\right).
\end{split}
\end{equation*}
Clearly, $f(0,k) = \mathfrak{S}(k)$ and $f(s,k)$ has no poles with $\Re(s) >0$.  Set
\[ b_q = \frac{\mu(q)}{\varphi(q)} \prod_{p|q}(n_{k,p}-1) \]
so that the Dirichlet series generated by $b_q$ is precisely $f(s,k)$. \newline

Consider the Hecke $L$-function
\begin{eqnarray*} &&L\left(s+1,\left(\frac{k}{\cdot}\right)_3\right) = \prod_{\pi}\left(1-\frac{\left(\frac{k}{\pi}\right)_3}{N(\pi)^{s+1}}\right)^{-1}\nonumber\\
&=& \left(1-\frac{\left(\frac{k}{(1-e(1/3))}\right)_3}{3^{s+1}}\right)^{-1}\prod_{\substack{p \equiv 1 \bmod 3\\p = \pi_{p,1}\pi_{p,2}}}\left(1-\frac{\left(\frac{k}{\pi_{p,1}}\right)_3}{p^{s+1}}\right)^{-1} \left(1-\frac{\left(\frac{k}{\pi_{p,2}}\right)_3}{p^{s+1}}\right)^{-1}\prod_{p \equiv 2 \bmod 3}\left(1-\frac{\left(\frac{k}{p}\right)_3}{p^{2s+2}}\right)^{-1},
\end{eqnarray*}
where the product is over all prime ideals $\pi$ of $\intz[\omega]$.  This is a Hecke $L$-function.  Set
\[ f(s,k) = L^{-1}\left( s+1,\left(\frac{k}{\cdot}\right)_3 \right)h(s,k) . \]
Standard arguments will show that $h(s,k)$ is bounded by an absolute constant for $\Re (s) > -1/2 + \delta$ for any fixed $\delta >0$. \newline

Using Perron's formula, Lemma~\ref{perron}, we have for any $C>0$,
\begin{equation} \label{8}
\begin{split}
 \sum_{y_1 \leq q \leq y_2}b_q = \frac{1}{2\pi i}\int\limits_{C-iT}^{C+iT}& L^{-1}\left(s+1,\left(\frac{k}{\cdot}\right)_3\right)h(s,k)\frac{y_2^s - y_1^s}{s}ds \\
 & + O\left( \sum_{j=1}^2 \sum_{q}|b_q|\left( \frac{y_j}{q} \right)^C\min \left( 1,T^{-1} \left| \log \frac{y_j}{q} \right|^{-1} \right) \right)
\end{split}
\end{equation}

Now shift the line of integration to $\sigma-iT$ and $\sigma + iT$, with
\[ \sigma = -\frac{c'}{2 d\log \Delta(|T|+3)} \]
with $c'$, $d$ and $\Delta$ as given in Lemma~\ref{zerofree}.  Note that the same lemma also ensures that we do not pick up any residues from the zeros of the Hecke $L$-function.  Moreover, the residue at $s=0$ is cancelled out.  We shall only estimate the contributions from the terms with $y_1$.  The terms with $y_2$ can be treated in the same way and satisfy similar bounds.  The integral along the vertical line segment $l$ from $\sigma-iT$ to $\sigma + iT$ is majorized by
\begin{equation}\label{9}
\sup_{s \in l}\left|L^{-1}\left(s,\left(\frac{k}{\cdot}\right)_3\right)\right|\log(\Delta |T|+3) \exp \left( -\frac{1 \log y_1}{2 \log (\Delta|T|+3)} \right) \ll \exp \left( -c_1 \sqrt{\log y_1} \right),
\end{equation}
for some fixed $c_1 >0$.  The last inequality above comes from setting $T$ to satisfy $\log y_1 = (\log T)^2$ and the fact that the inverse of the $L$-function on that line segment is bounded by a fixed power of $\log T$ using an analysis similar to that for the inverse of the Riemann zeta function in some zero free region found in Theorems 3.10 and 3.11 of \cite{ET}. \newline

The integral along the horizontal line segments $l_{\pm}$ from $\sigma\pm iT$ to $C\pm iT$ is majorized by
\begin{equation}\label{10}
\frac{ y_1^C \sup_{s \in l_{\pm}}\left|L^{-1}\left(s,\left(\frac{k}{}\right)_3\right)\right|}{T} \ll \exp \left( -c_2\sqrt{\log y_1} \right),
\end{equation}
with some $c_2 >0$, upon setting $C= 1/\log y_1$.  Futhermore, with our choices of $C$ and $T$, the $O$-term in \eqref{8} is majorized by
\begin{equation} \label{Otermest}
\exp \left( -c_3\sqrt{\log y_1} \right)
\end{equation}
for some $c_3>0$.  The terms with $y_2$ in \eqref{8} satisfies similar bounds.  \newline

Now, taking $y_1 = 2^v U$ and $y_2$ to infinity, we have that, upon combining \eqref{8}, \eqref{9}, \eqref{10} and \eqref{Otermest},
\begin{equation} \label{Psi3est}
 \Psi_3 = \sum_{k \leq y} \left| \sum_{q > 2^vU}\frac{\mu(q)}{\varphi(q)}\prod_{p|q}(\chi_{1,p}(-k)+\chi_{2,p}(-k)) \right|^2 \ll y \exp \left( -c_4 \sqrt{2^vU} \right),
 \end{equation}
for some $c_4 >0$. \newline

Now setting
\[ U=\sqrt{y}, \; v = \log_2\frac{\exp(y^{\varepsilon/2})}{U} \]
and combining \eqref{splitPhi}, \eqref{Psi1est}, \eqref{Psi2est} and \eqref{Psi3est}, we have
\begin{equation} \label{Phiest}
\sum_{\substack{k \leq y \\ \mu(k)^2 =1}} \left| \Phi(k) \right|^2 \ll \frac{y}{(\log x)^{c_5}},
 \end{equation}
for all $c_5 >0$.

\section{The Error Terms from the Major Arcs}

We now consider the second moment over $k$ of the error terms from the major arcs. By Bessel's inequality, Lemma~\ref{bessineq}, applied to the vector space $V=L^2([0,1])$ with $\xi(\alpha) = T_1(\alpha)E_2(\alpha)$ for $\alpha\in \mathfrak{M}$ and $\xi(\alpha) = 0$ for $\alpha \not\in \mathfrak{M}$, and $\phi_k=e(-\alpha k)$, we have
\begin{equation*}
\begin{split}
\sum_{k \leq y} & \left|\int\limits_{\mathfrak{M}}T_1(\alpha)E_2(\alpha)e(-\alpha k) \dif \alpha\right|^2 \ll \int\limits_{\mathfrak{M}} \left| T_1(\alpha)E_2(\alpha) \right|^2 \dif \alpha \\
&= \sum_{q < Q_1}\sum_{\substack{a \bmod q\\ \gcd(a,q)=1}} \left( \ \int\limits_{|\beta| < \delta}\left|T_1\left(\frac{a}{q} + \beta\right)E_2\left(\frac{a}{q}+\beta\right)\right|^2 \dif \beta + \int\limits_{\delta < |\beta| < 1/(qQ)}\left|T_1\left(\frac{a}{q}+\beta\right)E_2\left(\frac{a}{q}+\beta\right)\right|^2 \dif \beta \right) ,
\end{split}
\end{equation*}
where $\delta$ is to be chosen later.  Because $T_1(\alpha)$, defined in \eqref{T1E1def}, involves a geometric sum, it can be observed that
\[ T_1 \left( \frac{a}{q} + \beta \right) \ll \min \left( z, |\beta|^{-1} \right). \]

Therefore,
\begin{equation} \label{T1E2aftersplit}
\begin{split}
\sum_{k \leq y} & \left|\int\limits_{\mathfrak{M}}T_1(\alpha)E_2(\alpha)e(-\alpha k) \dif \alpha\right|^2 \\
& \ll  z^2 \sum_{q < Q_1}\sum_{\substack{a \bmod q\\ \gcd(a,q)=1}} \ \int\limits_{|\beta| < \delta}\left|E_2\left(\frac{a}{q}+\beta\right)\right|^2 \dif \beta + \delta^{-2} \int\limits_{\majorarc}\left|E_2\left(\frac{a}{q}+\beta\right)\right|^2 \dif \beta .
\end{split}
\end{equation}
Performing computations similar to those on pages 978 of \cite{SBLZ} to estimate the integral on the right-hand side of \eqref{T1E2aftersplit}, we have that the right-hand side of \eqref{T1E2aftersplit} is
\begin{equation} \label{T1E2intest}
 \ll z^2 x (\log x)^{c_6} \delta^2 + \frac{x (\log x)^{c_7}}{\delta^2 Q^2},
\end{equation}
for some fixed positive $c_6$ and $c_7$.  Now, taking $\delta = 1/\sqrt{z}$, \eqref{T1E2aftersplit} and \eqref{T1E2intest} give
\begin{equation} \label{T1E2est}
\sum_{k \leq y}  \left|\int\limits_{\mathfrak{M}}T_1(\alpha)E_2(\alpha)e(-\alpha k) \dif \alpha\right|^2 \ll zx(\log x)^{c_6} + \frac{zx(\log x)^{c_7}}{Q^2}.
\end{equation}

Again, by Bessel's inequality, Lemma~\ref{bessineq}, we have
\begin{equation} \label{T2E1bess}
\begin{split}
 \sum_{k \leq y}& \left|\int\limits_{\mathfrak{M}}T_2(\alpha)E_1(\alpha)e(-\alpha k) \dif \alpha\right|^2 \\
 & \ll \int\limits_{\mathfrak{M}}|T_2(\alpha)E_1(\alpha) |^2 \dif \alpha \ll \sup_{\alpha \in \mathfrak{M}}|T_2(\alpha)|^2\int\limits_{\mathfrak{M}}|E_1(\alpha)|^2 \dif \alpha \ll x^2\int\limits_{\mathfrak{M}}|E_1(\alpha)|^2 \dif \alpha.
 \end{split}
 \end{equation}
In a manner similar to the estimate of the analogous terms in \cite{SBLZ}, we get, using Lemma~\ref{mikawalem}, that
\begin{equation} \label{4}
\int\limits_{\mathfrak{M}}|E_1(\alpha)|^2 \dif \alpha \ll \sum_{q< Q_1}\frac{q}{\varphi(q)}(qQ)^{-2}\mathfrak{J}(q,Q/2) + Q_1^3Q(\log x)^2 \ll \sum_{q<Q_1} \frac{q}{\varphi(q)}z(\log z)^{-A},
\end{equation}
for any $A>0$, with the implied constant depending on $A$.  Here $\mathfrak{J}$ is the same as that appearing in Lemma~\ref{mikawalem}.  Thus, \eqref{T2E1bess} and \eqref{4} give that
\begin{equation} \label{T2E1est}
\sum_{k \leq y}\left|\int\limits_{\mathfrak{M}}T_2(\alpha)E_1(\alpha)e(-\alpha k) d\alpha\right|^2 \ll \frac{x^2z}{(\log x)^{c_8}}
\end{equation}
for any $c_8 > 0$. \newline

Finally, using Bessel's inequality, Lemma~\ref{bessineq}, \eqref{4} and the trivial estimate for $E_2(\alpha)$, we have
\begin{equation} \label{E1E2est}
 \sum_{k \leq y}\left|\int\limits_{\mathfrak{M}}E_1(\alpha)E_2(\alpha)e(-\alpha k) \dif \alpha\right|^2 \ll \int\limits_{\mathfrak{M}}|E_1(\alpha)E_2(\alpha) |^2 \dif \alpha \ll \sup_{\alpha \in \mathfrak{M}}|E_2(\alpha)|^2\int\limits_{\mathfrak{M}}|E_1(\alpha) |^2 \dif \alpha \ll \frac{x^2 z}{(\log x)^{c_9}}
 \end{equation}
for any $c_9 > 0$, with the implied constant depending on $c_5$. \newline

Combining \eqref{T1E2est}, \eqref{T2E1est} and \eqref{E1E2est}, we have that
\begin{equation} \label{majorarcerrest}
\sum_{k \leq y} \left|\int\limits_{\mathfrak{M}}(T_1(\alpha)E_2(\alpha) + T_2(\alpha)E_1(\alpha) + E_1(\alpha)E_2(\alpha))e(-\alpha k) d\alpha\right|^2 \ll zx(\log x)^{c_6} + \frac{zx(\log x)^{c_7}}{Q^2} + \frac{x^2z}{(\log x)^{c_{10}}},
\end{equation}
for any $c_{10} >0$ with the implied constant depending on $c_{10}$.

\section{The Minor Arcs}

Finally, it still remains to consider the contribution from the minor arcs.  We have
\begin{equation} \label{minorarc}
 \sum_{\substack{k \leq y\\ \mu^2(k) = 1}}\left| \int\limits_{\minorarc} S_1(\alpha) S_2(\alpha)e(-\alpha k) \dif \alpha \right|^2=  \sum_{\substack{k \leq y\\ \mu^2(k) = 1}}\left|\int\limits_{\minorarc} \sum_{m \leq z}\Lambda(m)e(\alpha m)\sum_{n \leq x}e(-\alpha(n^3 + k)) \dif \alpha\right|^2,
 \end{equation}
where
\[ \minorarc = \left[ \frac{1}{Q} , 1+ \frac{1}{Q} \right] - \majorarc. \]
By Bessel's inequality, Lemma~\ref{bessineq}, \eqref{minorarc} is
\begin{equation} \label{minorarcbess}
 \ll \int\limits_{\mathfrak{m}}|S_1(\alpha)S_2(\alpha)|^2 \dif \alpha \ll \sup_{\alpha \in \mathfrak{m}}|S_2(\alpha)|^2\int\limits_0^1\left| S_1(\alpha) \right|^2 \dif \alpha \ll \sup_{\alpha \in \mathfrak{m}}|S_2(\alpha)|^2 z\log z.
 \end{equation}

Using Weyl shift, Lemma~\ref{weylshift}, we get
\begin{equation} \label{minorarcweyl}
 S_2^4 (\alpha)  \ll x \underset{-x < r_1,r_2 < x}{\sum\sum} \min\left(x, \frac{1}{\| 6\alpha r_1r_2\|}\right).
 \end{equation}

By Dirichlet approximation, there exists a rational approximation to $\alpha$ of type
\[ \left|\alpha - \frac{a}{q}\right| \leq \frac{1}{12x^2q} \]
with $\gcd(a,q) = 1$ and $1 \leq q \leq 12x^2$. But since $\alpha \in \mathfrak{m}$, we can assume that $q > Q_1$. Hence for $-x < r_1,r_2 < x$,
\[ \left|6r_1r_2\alpha - 6r_1r_2\frac{a}{q}\right| \leq \frac{1}{2q} \]
which yields
\[ \frac{1}{\|6r_1r_2\alpha\|} \leq \frac{2}{\|6r_1r_2a/q\|}. \]
We have
\begin{equation} \label{minorarcafterweyl}
 \underset{-x < r_1,r_2 < x}{\sum\sum}\min\left(x, \frac{1}{\| 6 \alpha r_1r_2\|}\right)
\leq x \sum_{\substack{-x < r_1,r_2 < x\\q |6r_1r_2}} 1 + \sum_{\substack{-x < r_1,r_2 <x\\q \nmid 6r_1r_2 }}\frac{2}{\| 6 (a/q) r_1r_2\|}.
\end{equation}
Now, letting $q^{\prime} = q/\gcd(q,6)$ and $d(n)$ denote the number of divisors of $n$, we get
\begin{equation} \label{weyl1sttermest}
\sum_{\substack{-x < r_1,r_2 < x\\q |6r_1r_2}} 1 \ll \sum_{d|q^{\prime}} \left(\frac{x}{d} + 1\right)\left(\frac{x}{q^{\prime}/d} + 1 \right) \ll \left(d(q^{\prime})\frac{x^2}{q^{\prime}} + d(q^{\prime})x \right) \ll \frac{x^2}{(\log x)^{c_{11}}}
\end{equation}
for any $c_{11} >0$ and
\begin{equation} \label{weyl2ndtermest}
\sum_{\substack{-x < r_1,r_2 <x\\q \nmid 6r_1r_2 }}\frac{2}{\| 6(a/q) r_1r_2\|} \ll x^{2 + \varepsilon}.
\end{equation}

Therefore, combining \eqref{minorarcweyl}, \eqref{minorarcafterweyl}, \eqref{weyl1sttermest} and \eqref{weyl2ndtermest} gives
\[ \sup_{\alpha \in \mathfrak{m}}\left| S_2(\alpha) \right|^2 \ll \frac{x^2}{(\log x)^{c_{12}}}, \]
for any $c_{12} >0$.  Inserting the above into \eqref{minorarcbess}, we have
\begin{equation} \label{minorarcest}
\sum_{\substack{k \leq y\\ \mu^2(k) = 1}}\left| \int\limits_{\minorarc} S_1(\alpha) S_2(\alpha)e(-\alpha k) \dif \alpha \right|^2 \ll \frac{x^5}{(\log x)^{c_{13}}}
\end{equation}
for any $c_{13} >0$.

\section{Proof and Discussion of the Theorem}

\begin{proof} [Proof of the Theorem]
Using Cauchy's inequality, combining \eqref{startpt}, \eqref{majorarc1}, \eqref{mainterm}, \eqref{Phiest}, \eqref{majorarcerrest} and \eqref{minorarcest}, we obtain the theorem.
\end{proof}

It should be observed that the Theorem is an analogue of the Theorem in \cite{SBLZ}.  The latter was improved in \cite{BZ} using a variant of the dispersion method of J. V. Linnik \cite{Linnik}.  It would be highly desirable to have a similar improvement for the cubic polynomials.  However, this proved to be not at all straightforward, as the condition $m_1-m_2=n_1^2-n_2^2$ appearing in section 5 of \cite{BZ}, if generalized to the cubic case, would not yield such an easy correspondence to work with between $(m_1,m_2)$ and $(n_1,n_2)$ as the quadratic case does. \newline

It is natural to ask whether this method can be used for other families of polynomials, such as those of higher degree or other families of cubic polynomials. Observe that the treatment of $\Psi(k)$, defined in \eqref{Psikdef}, requires explicit expression of the numbers $n_p$ in terms of characters that reflect, in this case, the arithmetic of $\mathbb{Q}\left[ \omega, k^{1/3}\right]$, since that is the splitting field of $x^3+k$. Analogues of this will be necessary regardless of what family of polynomials we are considering. For the family $x^d+k$, $d\in \mathbb{Z}$, $d>3$ where $k$ runs in such a way that the polynomials  $x^d+k$ are irreducible, this can be quite easily mimicked, and the characters in question are power residue symbols governing the Kummer extension $\mathbb{Q}\left[e\left(1/d\right),k^{1/d}\right]/\mathbb{Q}\left[e\left(1/d\right)\right]$. However, analogues of Lemma~\ref{B&Y} are missing in all these cases except $d=4$, in which case that Lemma can be replaced by Theorem 1.2 of \cite{GZ}. The situation with other families of cubic polynomials is slightly different. The discriminant of a cubic polynomial $f(x)=x^3+ax^2+bx+c$ is $D = 18abc+a^2b^2-4b^3-4a^3c-27c^2$, and the splitting field of $f$ contains $\mathbb{Q}[\sqrt{-D}]$ as a subfield. This field is always $\mathbb{Q}\left[ \omega \right]$ for the family of irreducible $x^3+k$, allowing the use of Lemma~\ref{huxleylemma} with that field. For $a\not=0$, $b\not=0$, any family of cubic polynomials with fixed $a$ and $b$ and varying $c$, no longer has a uniform $\mathbb{Q}[\sqrt{-D}]$ for the use of that Lemma, while attempts to create a family with a fixed $\mathbb{Q}[\sqrt{-D}]$, appear to preclude a straightforward use of Lemma~\ref{bessineq} in treating the error terms from the major arcs.\newline

\noindent{\bf Acknowledgements.}  The authors were supported by an AcRF Tier 1 grant at Nanyang Technological University during this work.

\bibliography{biblio}
\bibliographystyle{amsxport}

\vspace*{.7cm}

\noindent\begin{tabular}{p{8cm}p{8cm}}
Div. of Math. Sci., School of Phys. \& Math. Sci., & Div. of Math. Sci., School of Phys. \& Math. Sci., \\
Nanyang Technological Univ., Singapore 637371 & Nanyang Technological Univ., Singapore 637371 \\
Email: {\tt S080074@ntu.edu.sg} & Email: {\tt lzhao@pmail.ntu.edu.sg} \\
\end{tabular}

\end{document}